\documentclass[11pt,a4paper]{amsart}
\usepackage{amssymb,amsmath}
\usepackage[dvips]{color}
\usepackage[dvips,final]{graphics}
\usepackage{epstopdf}
\usepackage{amscd}

\usepackage{pdfsync}

\theoremstyle{plain}
\newtheorem{theorem}{Theorem}[section]
\newtheorem{proposition}[theorem]{Proposition}
\newtheorem{corollary}[theorem]{Corollary}

\newtheorem{lemma}[theorem]{Lemma}
\theoremstyle{definition}
\newtheorem{definition}[theorem]{Definition}

\theoremstyle{remark}
\newtheorem{remark}[theorem]{Remark}

\begin{document}

\bibliographystyle{plain}

\newcommand{\Rn}{\mathbb R^n}
\newcommand{\E}{\mathbb E}
\newcommand{\Rm}{\mathbb R^m}
\newcommand{\rn}[1]{{\mathbb R}^{#1}}
\newcommand{\R}{\mathbb R}
\newcommand{\C}{\mathbb C}
\newcommand{\G}{\mathbb G}
\newcommand{\M}{\mathbb M}
\newcommand{\Z}{\mathbb Z}
\newcommand{\D}[1]{\mathcal D^{#1}}
\newcommand{\Ci}[1]{\mathcal C^{#1}}
\newcommand{\Ch}[1]{\mathcal C_{\mathbb H}^{#1}}
\renewcommand{\L}[1]{\mathcal L^{#1}}
\newcommand{\BVG}{BV_{\G}(\Omega)}
\newcommand{\supp}{\mathrm{supp}\;}

\newcommand{\dom}{\mathrm{Dom}\;}
\newcommand{\Ast}{\un{\ast}}

\newcommand{\average}{{\mathchoice {\kern1ex\vcenter{\hrule height.4pt
width 6pt depth0pt} \kern-9.7pt} {\kern1ex\vcenter{\hrule height.4pt width
4.3pt depth0pt}
\kern-7pt} {} {} }}
\newcommand{\ave}{\average\int}

\newcommand{\hhd}[1]{{\mathcal H}_d^{#1}}
\newcommand{\hsd}[1]{{\mathcal S}_d^{#1}}

\newcommand{\he}[1]{{\mathbb H}^{#1}}
\newcommand{\hhe}[1]{{H\mathbb H}^{#1}}

\newcommand{\cov}[1]{{\bigwedge\nolimits^{#1}{\mfrak g}}}
\newcommand{\vet}[1]{{\bigwedge\nolimits_{#1}{\mfrak g}}}

\newcommand{\covw}[2]{{\bigwedge\nolimits^{#1,#2}{\mfrak g}}}

\newcommand{\vetfiber}[2]{{\bigwedge\nolimits_{#1,#2}{\mfrak g}}}
\newcommand{\covfiber}[2]{{\bigwedge\nolimits^{#1}_{#2}{\mfrak g}}}

\newcommand{\covwfiber}[3]{{\bigwedge\nolimits^{#1,#2}_{#3}{\mfrak g}}}

\newcommand{\covv}[2]{{\bigwedge\nolimits^{#1}{#2}}}
\newcommand{\vett}[2]{{\bigwedge\nolimits_{#1}{#2}}}

\newcommand{\covvfiber}[3]{{\bigwedge\nolimits^{#1}_{#2}{#3}}}

\newcommand{\vettfiber}[3]{{\bigwedge\nolimits_{#1,#2}{#3}}}

\newcommand{\covn}[2]{{\bigwedge\nolimits^{#1}\rn {#2}}}
\newcommand{\vetn}[2]{{\bigwedge\nolimits_{#1}\rn {#2}}}
\newcommand{\covh}[1]{{\bigwedge\nolimits^{#1}{\mfrak h_1}}}
\newcommand{\veth}[1]{{\bigwedge\nolimits_{#1}{\mfrak h_1}}}
\newcommand{\hcov}[1]{{_H\!\!\bigwedge\nolimits^{#1}}}
\newcommand{\hvet}[1]{{_H\!\!\bigwedge\nolimits_{#1}}}

\newcommand{\covf}[2]{{\bigwedge\nolimits^{#1}_{#2}{\mfrak h}}}
\newcommand{\vetf}[2]{{\bigwedge\nolimits_{#1,#2}{\mfrak h}}}
\newcommand{\covhf}[2]{{\bigwedge\nolimits^{#1}_{#2}{\mfrak h_1}}}
\newcommand{\vethf}[2]{{\bigwedge\nolimits_{#1,#2}{\mfrak h_1}}}
\newcommand{\hcovf}[2]{{_H\!\!\bigwedge\nolimits^{#1}_{#2}}}
\newcommand{\hvetf}[2]{{_H\!\!\bigwedge\nolimits_{#1,#2}}}

\newcommand{\defin}{\stackrel{\mathrm{def}}{=}}
\newcommand{\gradh}{\nabla_H}
\newcommand{\current}[1]{\left[\!\left[{#1}\right]\!\right]}
\newcommand{\scal}[2]{\langle {#1} , {#2}\rangle}
\newcommand{\escal}[2]{\langle {#1} , {#2}\rangle_{\mathrm{Euc}}}
\newcommand{\Scal}[2]{\langle {#1} \vert {#2}\rangle}
\newcommand{\scalp}[3]{\langle {#1} , {#2}\rangle_{#3}}
\newcommand{\dc}[2]{d_c\left( {#1} , {#2}\right)}
\newcommand{\res}{\mathop{\hbox{\vrule height 7pt width .5pt depth 0pt
\vrule height .5pt width 6pt depth 0pt}}\nolimits}
\newcommand{\norm}[1]{\left\Vert{#1}\right\Vert}
\newcommand{\modul}[2]{{\left\vert{#1}\right\vert}_{#2}}
\newcommand{\perh}{\partial_\mathbb H}

\newcommand{\ccheck}{{\vphantom i}^{\mathrm v}\!\,}

\newcommand{\wcheck}{{\vphantom i}^{\mathrm w}\!\,}

\newcommand{\mc}{\mathcal }
\newcommand{\mbf}{\mathbf}
\newcommand{\mfrak}{\mathfrak}
\newcommand{\mrm}{\mathrm}
\newcommand{\no}{\noindent}
\newcommand{\dis}{\displaystyle}

\newcommand{\U}{\mathcal U}
\newcommand{\ga}{\alpha}
\newcommand{\gb}{\beta}
\newcommand{\gga}{\gamma}
\newcommand{\gd}{\delta}
\newcommand{\eps}{\varepsilon}
\newcommand{\gf}{\varphi}
\newcommand{\GF}{\varphi}
\newcommand{\gl}{\lambda}
\newcommand{\GL}{\Lambda}
\newcommand{\gp}{\psi}
\newcommand{\GP}{\Psi}
\newcommand{\gr}{\varrho}
\newcommand{\go}{\omega}
\newcommand{\gs}{\sigma}
\newcommand{\gt}{\theta}
\newcommand{\gx}{\xi}
\newcommand{\GO}{\Omega}

\newcommand{\Wedge}{\buildrel\circ\over \wedge}

\newcommand{\WO}[4]{\mathop{W}\limits^\circ{}\!_{#4}^{{#1},{#2}}
(#3)}

\newcommand{\GH}{H\G}
\newcommand{\N}{\mathbb N}

%

\newcommand{\Nhmin}{N_h^{\mathrm{min}}}
\newcommand{\Nhmax}{N_h^{\mathrm{max}}}

\newcommand{\Mhmin}{M_h^{\mathrm{min}}}
\newcommand{\Mhmax}{M_h^{\mathrm{max}}}

\newcommand{\un}[1]{\underline{#1}}

\newcommand{\curl}{\mathrm{curl}\;}
\newcommand{\curlh}{\mathrm{curl}_{\he{}}\;}
\newcommand{\hd}{\hat{d_c}}
\newcommand{\divg}{\mathrm{div}_\G\,}
\newcommand{\divgh}{\mathrm{div}_{\hat\G}\,}
\newcommand{\divh}{\mathrm{div}_{\he{}}\,}
\newcommand{\e}{\mathrm{Euc}}


\newcommand{\hatcov}[1]{{\bigwedge\nolimits^{#1}{\hat{\mfrak g}}}}

\title[Harnack inequality for fractional sub-Laplacians in Carnot groups
] {Harnack inequality for fractional sub-Laplacians in Carnot groups
}

\author[ F. Ferrari, B. Franchi]{Fausto Ferrari
\endgraf \medskip \endgraf Bruno Franchi }

\thanks{The authors are supported by  MURST, Italy, and
by University of Bologna, Italy, funds for selected research topics
and by EC project CG-DICE}

\thanks{}
%

%

%


\begin{abstract} 
In this paper we prove an invariant Harnack inequality on Carnot-Carath\'eodory balls
for fractional powers of sub-Laplacians in Carnot groups. The proof relies on an ``abstract'' formulation of
a technique recently introduced by Caffarelli and Silvestre. In addition, we write
explicitly the Poisson kernel for a class of degenerate subelliptic equations in
product-type Carnot groups.
\end{abstract}

\maketitle

\tableofcontents

\section{Introduction}\label{introduction} In Euclidean spaces, fractional operators have been 
studied in connection with different phenomena that can be described as isotropic diffusion
with jumps. We mention, for instance, the thin obstacle problem, 
phase transition problems, and  the study of a general class of conformally covariant operators
in conformal geometry: see, for instance, \cite{caffarelli_salsa_silvestre},  \cite{sire_valdinoci} and \cite{chang_gonzalez}. 
Typically, these problems can be reduced, in their simplest form, to the study of the
equation
\begin{equation}\label{itro eq:1}
(-\Delta)^{\gamma/2} u = f  \quad\mbox{in $\rn n$,}
\end{equation}
where $0<\gamma<2$. We remind that the fractional Laplacian in \eqref{itro eq:1} is
a nonlocal operator (even more: it is a {\sl antilocal operator}, see \cite{masuda}).
Nevertheless, solutions of \eqref{itro eq:1} share some properties of the solutions
of elliptic equations. More precisely:
\begin{itemize}
\item $(-\Delta)^{\gamma/2} $ is the infinitesimal generator of
 a Feller semigroup $\{T_t\}_{t>0}$. This
means that, if $0\le f\le 1$, then $0\le T_tf \le 1$ for $t>0$. By a classical result (see P. L\'evy \cite{levy},
G. A. Hunt \cite{hunt}, Courr\`ege \cite{courrege} and Bony-Courr\`ege-Priouret 
\cite{bony_courrege_priouret}), this is equivalent to say that
$(-\Delta)^{\gamma/2} $ belongs to a larger class of pseudodifferential operators satisfying
 the so-called {\sl positive maximum principle}. We refer
to \cite{courrege} and \cite{bony_courrege_priouret} for an exhaustive discussion; here
we restrict ourselves to stress that the positive maximum principle is not the usual
maximum principle of potential theory. 
\item recently, L. Caffarelli \& L. Silvestre \cite{caffarelli_silvestre}  proved that functions
$u$ that are positive {\sl on all of $\rn n$} and
solve the equation $(-\Delta)^{\gamma/2}u =0 $  {\sl in an open set $\Omega\subset\rn n$} 
satisfy an invariant local Harnack inequality.   Their technique relies on an extension (or
`lifting'')  procedure, showing ultimately that $u$ can be
extended to a function $\tilde v$ on $\rn{n+1}$ satisfying a (degenerate) elliptic {\it differential}
equation.

We remind also that related results have been proved by different methods by N.S. Landkof \cite{landkof} and
K. Bogdan \cite{bogdan}.
\end{itemize}
On the other hand, 
\begin{itemize}
\item Hunt's theorem in \cite{hunt} applies to a larger class of differential
operators in Lie groups;
\item sub-Laplacians in Carnot groups (i.e. in connected and simply connected
stratified nilpotent Lie groups) exhibit strong analogies with classical Laplace
operator in the Euclidean space (for instance Harnack inequality, maximum principle, 
existence and estimates of the fundamental solution). 
\end{itemize}
It is therefore natural to ask whether Caffarelli \& Silvestre's approach can be adapted
to prove a Harnack inequality for subelliptic fractional equations of the form
$$
\mc L^{\gamma/2} u = 0,
$$
where $\mc L$ is a (positive) sublaplacian in a Carnot group $\G$. 

 In fact, an ``abstract'' extension technique akin to that of Caf\-fa\-rel\-li-Silvestre has been recently developed
in a general setting by Stinga \& Torrea in \cite{Stinga_Torrea}, under very mild hypotheses 
on the operator $\mathcal{L}$. In particular, they obtained the Harnack inequality 
for the (fractional) harmonic oscillator. In addition, using analogous arguments, Stinga \& Zhang \cite{Stinga_Zhang} proved a Harnack inequality  
for  a larger class of fractional operators, containing, for instance, Ornstein-Uhlenbeck operators. 
However, we stress that
subelliptic operators in Carnot groups, though, as a matter of fact, fitting in the wide class of
``degenerate elliptic operators'', do not belong to the class of degenerate
operators considered in \cite{Stinga_Zhang}. Indeed, the degeneration considered in \cite{Stinga_Zhang} is
described by means of $A_2$-weights that may vanish only on sets of finite
Lebesgue measure. On the contrary, subelliptic Laplacians, when considered as
degenerate elliptic operators,  may in fact degenerate on all the space. 
In other words, the degeneration induced by weights is a ``degeneration of the measure'', whereas
subelliptic Laplacians could be considered as Laplace-Beltrami operators for a degenerate
geometry. 

Typically, if we forget the potentials, operators as in \cite{Stinga_Zhang} have the form
$$
\big(-\mathrm{div}\,(|x|^\alpha \nabla u)\big)^{\gamma/2}, \quad -n < \alpha < n
$$
in $\rn n$, whereas, the simplest instance of our operators is provided by the fractional sub-Laplacian
of the first Heisenberg group $\he 1$
$$
\big( -(\partial_x +2y\partial_z)^2 u - (\partial_y -2x\partial_z)^2 u\big)^{\gamma/2}
$$
in $\rn 3$. Some comments in this sense can be found already in \cite{FS_ASNP}.

  In this paper
we further develop the idea of an abstract approach to the problem.
However, the setting of Carnot groups,
with a natural notion of group convolution, makes possible to recover, starting from
the abstract representation in terms of the spectral resolution, another explicit
form of the fractional powers (in terms of convolutions with singular kernels),
as well as of the lifting operator (in terms of the convolution with a suitable Poisson kernel).

We like also to mention  that, in the special case of Heisenberg groups, an explicit
representation of the Poisson kernel is given also in \cite{gonzalez_tan} through
different methods (group Fourier transform).

To state our main result,
we need preliminarily to remind that in any Carnot group we can define a left-invariant
distance $d_c$ (the so-called Carnot-Carath\'eodory distance) that fits the structure of
the group. If we denote by $B_c = B_c(x,r)$ ($x\in \G$ and $r>0$) the metric balls associated
with $d_c$ and by $W^{s,2}_\G$ the Folland-Stein Sobolev space in $\G$ (see Section \ref{introduction}
for details), then the Harnack inequality for fractional sub-Laplacians takes an
invariant intrinsic form. More precisely, we have:

\medskip

\noindent\textbf{Theorem}
\textsl{ 
Let $-1<a<1$ and let $u\in W^{1-a,2}_\G(\G)$ be given, $u\ge 0$ on all of $\G.$ Assume $\mc L^{(1-a)/2}u = 0$
in an open set $\Omega\subset\G$. 
\\
\hphantom{xxx}Then  there exist $C,b>0$ (independent of $u$) such that
the following invariant Harnack inequality holds:
$$
\sup_{B_c(x,r)} u \le C \inf_{B_c(x,r)} u 
$$
for any metric ball $B_c(x,r)$ such that $B_c(x,b r)\subset \Omega$. 
}

\medskip

Let us sketch briefly the main features of our proof. Basically, still following \cite{caffarelli_silvestre},
its core consists in the construction of a $\mc L$-harmonic ``lifting'' operator $u=u(x)\to v=v(x,y)$ from $\G$ to $\G\times\R^+$ 
by means of the spectral resolution of $\mc L$ in $L^2(\G)$ in such a way that $u$ is the trace of the normal derivative of $v$ on $y=0$.  If, in particular, $a=0$, then this operator
is nothing but the semigroup generated by $-\mc L^{1/2}$.

 Subsequently, as in \cite{caffarelli_silvestre}, we show that, if $\mathcal{L}^{\frac{1-a}{2}}u=0$ in an open set $\Omega$ then its lifting $v$ can be continued  by parity across $y=0$  to  a weak solution $\tilde{v}$ of the equation
 $$
\tilde{\mathcal{L}}\tilde{v}:=-|y|^a\mc L \tilde{v} +\partial_y (|y|^a \partial_y \tilde{v})=0.
 $$ 
 In addition we show that the lifting operator  can be
also written as a convolution operator with a positive kernel $P_{\G}$, that is written explicitly. Thus $\tilde{v}\geq 0$  if  $u\geq 0$ on all $\G,$ and therefore  our problem  reduces to prove Harnack inequality for a weighted sub-elliptic differential  operator.
The construction of $P_{\G}$  not only yields the possibility of replacing the assumption $u\in W^{1-a,2}_\G(\G)$ by some weaker
assumptions on the behavior of $u$ at infinity (in the spirit of some remarks in \cite{caffarelli_silvestre}), but 
provides an explicit form for the Poisson kernel $P_\G(\cdot,y)$
 in the half-space $\G\times (0,\infty)$ for 
 $
\tilde{\mathcal{L}}.
 $ 
More precisely, if we denote by $h(t,\cdot)$ the heat kernel associated with $-\mc L$ as in \cite{folland}, then
\begin{equation*}
P_\G(\cdot,y) : = C_a\, y^{1-a} \int_0^\infty t^{(a-3)/2} e^{- \frac{y^2}{4t}}  h(t,\cdot)\; dt,
\end{equation*}
where
$$C_a = {2^{a-1}}{\Gamma((1-a)/2)}^{-1}.
$$
 A similar formula appears in \cite{Stinga_Torrea}, but, as long as we know, this representation is new
 for sublaplacians in Carnot groups.
 
The paper is organized as follows: in Section \ref{preliminary} we fix our notations
for Carnot groups and for Harnack inequality  in this setting; in Section
\ref{fractional} we collect some more or less known results on fractional
powers of sub-Laplacian in Carnot groups and we prove different representation
theorems. Finally, in Section \ref{main} we prove our main results.

\section{Preliminary results}\label{preliminary}

A connected and simply connected
 Lie group $(\G,\cdot)$  (in general non-com\-mu\-ta\-tive) is said a {\it Carnot group  of
step $\kappa$} if its Lie algebra
${\mathfrak{g}}$  admits a {\it step $\kappa$ stratification}, i.e.
there exist linear subspaces $V_1,...,V_\kappa$ such that
\begin{equation}\label{stratificazione}
{\mathfrak{g}}=V_1\oplus...\oplus V_\kappa,\quad [V_1,V_i]=V_{i+1},\quad
V_\kappa\neq\{0\},\quad V_i=\{0\}{\textrm{ if }} i>\kappa,
\end{equation}
where $[V_1,V_i]$ is the subspace of ${\mathfrak{g}}$ generated by
the commutators $[X,Y]$ with $X\in V_1$ and $Y\in V_i$. 
The first layer $V_1$, the so-called 
 {\sl horizontal layer},  plays a
key role in the theory, since it generates  $\mathfrak g$ by commutation. 

For a general introduction to Carnot groups from the point of view of the
present paper, we refer, e.g., to \cite{BLU}, \cite{folland_stein} and \cite{stein}.

Set
$m_i=\dim(V_i)$, for $i=1,\dots,\kappa$ and $h_i=m_1+\dots +m_i$, so that $h_\kappa=n$. For sake of simplicity, we write also
$m:=m_1$.
We denote by $Q$ the {\sl homogeneous dimension} of $\G$, i.e. we set
 $$
 Q:=\sum_{i=1}^{\kappa} i \dim(V_i).
 $$
 
If $e$ is the unit element of $(\G,\cdot)$, we remind that the map $X\to X(e)$,
that associate with a left-invariant vector field $X$ its value at $e$, is an
isomorphism from $\mathfrak g$ to $T\G_e$, in turn identified with $\rn n$.
From now on, we shall use systematically these identications. Thus,
the horizontal layer defines, by left translation, a fiber bundle $H\G$ over $\G$
(the {\sl horizontal bundle}). Its sections are the {\sl horizontal vector fields}.

We choose now a basis $e_1,\dots,e_n$ of
$\rn n$ adapted to the stratification of $\mathfrak g$, i.e. such that
$$e_{h_{j-1}+1},\dots,e_{h_j}\;\text {is a basis of}\; V_j\;\text{
for each}\; j=1,\dots, \kappa.$$
Then, we denote by $\scal{\cdot}{\cdot}$ the scalar product in $\mathfrak g$ making 
the adapted basis $\{e_1,\dots,e_n\}$ orthonormal.
Moreover, let $X=\{X_1,\dots,X_{n}\}$ be the family
of left invariant vector fields such 
that 
$X_i(e)=e_i$, $i=1,\dots,n$. Clearly, $X$ is othonormal with respect
to $\scal{\cdot}{\cdot}$.

A Carnot group $\G$ can be always identified, through exponential coordinates,
 with the Euclidean space $(\rn n, \cdot)$,
where $n$ is the dimension of ${\mathfrak{g}}$, endowed with a suitable
group operation. The explicit
expression of the group operation $\cdot$ is determined by the
Campbell-Hausdorff formula.

For
any $x\in\G$, the {\it (left) translation} $\tau_x:\G\to\G$ is defined
as $$ z\mapsto\tau_x z:=x\cdot z. $$ For any $\lambda >0$, the
{\it dilation} $\delta_\lambda:\G\to\G$, is defined as
\begin{equation}\label{dilatazioni}
\delta_\lambda(x_1,...,x_n)=
(\lambda^{d_1}x_1,...,\lambda^{d_n}x_n),
\end{equation} where $d_i\in\N$ is called {\it homogeneity of
the variable} $x_i$ in
$\G$ (see \cite{folland_stein} Chapter 1) and is defined as
\begin{equation}\label{omogeneita2}
d_j=i \quad\text {whenever}\; h_{i-1}+1\leq j\leq h_{i},
\end{equation}
hence $1=d_1=...=d_{m_1}<
d_{{m_1}+1}=2\leq...\leq d_n=\kappa.$

Through this paper, by $\G$-homogeneity we mean homogeneity with respect
to group dilations $\delta_\lambda$ (see  again \cite{folland_stein} Chapter 1).

 The Haar measure of $\G=(\rn n,\cdot)$ is the Lebesgue measure
in $\rn n$. 
If $A\subset \G$ is $ L$-measurable, we
write $|A|$ to denote its Lebesgue measure.
Moreover, if $m\ge 0$, we denote by $\mathcal
H^m$ the $m$-dimensional Hausdorff measure obtained from the
Euclidean distance in $\R^n\simeq \G.$


The following result is contained in  \cite{folland_stein}, Proposition 1.26.
\begin{proposition}\label{campi omogenei0}
If $j=1,\dots,m$, the vector fields $X_j$ have polynomial coefficients and have the
form
\begin{equation}\label{campi omogenei}
X_j(x)=\partial_j+\sum_{d_k>1}^n p_{j,k}(x)\partial_k, \quad
\end{equation} where
the $p_{j,k}$ are $\G$-homogeneous polynomials of degree $d_k-1$
for $d_k > 1$.
\end{proposition}

Once a basis  $X_1,\dots,X_{m}$ of the horizontal layer is fixed, we
define,
for any function $f:\G\to \R$ for which the partial derivatives
$X_jf$ exist, the horizontal gradient of $f$, denoted by
$\nabla_{\G}f$, as the horizontal section
\begin{equation*}
\nabla_{\G}f:=\sum_{i=1}^{m}(X_if)X_i,
\end{equation*}
whose coordinates are $(X_1f,...,X_{m}f)$. Moreover, if
$\phi=(\phi_1,\dots,\phi_{m})$ is an horizontal section such
that $X_j\phi_j\in L^1_{\rm loc}(\G)$ for $j=1,\dots,m$, we define $\divg
\phi$ as the real valued function
\begin{equation*}
\divg(\phi):=-\sum_{j=1}^{m}X_j^*\phi_j=\sum_{j=1}^{m}X_j\phi_j.
\end{equation*}

Following \cite{folland_stein},
we also adopt the following multi-index notation for higher-order derivatives. If $I =
(i_1,\dots,i_{n})$ is a multi--index, we set  
$X^I=X_1^{i_1}\cdots
X_{n}^{i_{n}}$. By the Poincar\'e--Birkhoff--Witt theorem
(see, e.g. \cite{bourbaki}, I.2.7), the differential operators $X^I$ form a basis for the algebra of left invariant
differential operators in $\G$. Furthermore, we set 
$|I|:=i_1+\cdots +i_{n}$ the order of the differential operator
$X^I$, and   $d(I):=d_1i_1+\cdots +d_ni_{n}$ its degree of $\G$-homogeneity
with respect to group dilations.

 Let $X_1,\dots,X_m$ be a basis of the first layer of $\mathfrak g$, we denote by $\mc L$ the
associated positive sub-Laplacian
$$
\mc L:= -\sum_{j=1}^m X_j^2.
$$
It is easy to see that
$$
\mc Lu = - \divg (\nabla_\G u).
$$
In addition, $\mc L$ is left-invariant, i.e. for any $x\in\G$, we have
$$\mc L (u\circ \tau_x) = (\mc Lu)\circ \tau_x.$$

%
Following e.g. \cite{folland_stein}, we can define a group
convolution in $\G$: if, for instance, $f\in\mc D(\G)$ and
$g\in L^1_{\mathrm{loc}}(\G)$, we set
\begin{equation}\label{group convolution}
f\ast g(x):=\int f(y)g(y^{-1}x)\,dy\quad\mbox{for $x\in \G$}.
\end{equation}
We remind that, if (say) $g$ is a smooth function and $L$
is a left invariant differential operator, then
$
L(f\ast g)= f\ast Lg.
$
We remind also that the convolution is again well defined
when $f,g\in\mc D'(\G)$, provided at least one of them
has compact support (as customary, we denote by
$\mc E'(\G)$ the class of compactly supported distributions
in $\G$ identified with $\rn {n}$).

\medskip

If $E\subset\G$ is a mesurable set, a notion of $\G$-perimeter measure
$|\partial E|_\G$ has been introduced in \cite{GN}. We refer to \cite{GN},
\cite{FSSC_houston}, \cite{FSSC_CAG}, \cite{FSSC_step2} for a detailed presentation. For
our needs, we restrict ourselves to remind that, if $E$ has locally finite $\G$-perimeter
(is a $\G$-Caccioppoli set),
then $|\partial E|_\G$ is a Radon measure in $\G$, invariant under
group translations and $\G$-homogeneous of degree $Q-1$. Moreover, the following
representation theorem holds (see  \cite{capdangar}).

\begin{proposition}\label{perimetro regolare}
If $E$ is a $\G$-Caccioppoli set with Euclidean ${\mathbf C}^1$
boundary, then there is an explicit representation of the
$\G$-perimeter in terms of the Euclidean $(n-1)$-dimensional
Hausdorff measure $\mathcal H^{n-1}$
\begin{equation*}
|\partial E|_\G(\Omega)=\int_{\partial
E\cap\Omega}\bigg(\sum_{j=1}^{m_1}\langle
X_j,n\rangle_{\Rn}^2\bigg)^{1/2}d{\mathcal {H}}^{n-1},
\end{equation*}
where $n=n(x)$ is the Euclidean unit outward normal to $\partial
E$.
\end{proposition}

We have also
\begin{proposition}\label{divergence}
If $E$ is a regular bounded open set with Euclidean ${\mathbf C}^1$
boundary and $\phi$ is a horizontal vector field,
continuously differentiable on $ \overline{\Omega} $, then
$$
\int_E \divg \phi\, dx = \int_{\partial E} \scal{\phi}{n_\G} d |\partial E|_\G,
$$
where $n_\G(x)$ is the intrinsic horizontal outward normal to $\partial E$,
given by the (normalized) projection of $n(x)$ on the fiber $H\G_x$ of
the horizontal fibre bundle $H\G$.
\end{proposition}

\begin{remark}
The definition of $n_\G$ is well done, since $H\G_x$ is transversal to
the tangent space to $E$ at $x$ for  $|\partial E|_\G$-a.e. $x\in\partial E$
(see \cite{magnani}).
\end{remark}

\medskip

\begin{definition}\label{distanza}{\bf (Carnot-Carath\'eodory
distance)}
An absolutely continuous curve $\gamma:[0,T]\to \G$ is a {\it
sub-unit curve} with respect to $X_1,\dots,X_{m}$ if it is an
{\it horizontal curve}, i.e. if there are real measurable
functions $c_1(s),\dots,c_{m}(s)$, $s\in [0,T]$ such that
$$\dot\gamma(s)=\sum\limits_{j=1}^{m}\,c_j(s) X_j(\gamma(s)),
\qquad \text{for a.e.}\;s\in [0,T],$$ and if, in addition,
 $$\sum_jc_j^2\le 1.$$

If $x,y\in\G$,
their Carnot-Carath\'eodory
distance (cc-distance) $d_c(x,y)$ is defined as follows:
$$
d_c(x,y)=\inf\left\{T>0:\;\text{there is a subunit curve}\;
\gamma\;\text{with}\; \gamma(0)=x,\,\gamma(T)=y\right\}. $$
\end{definition}
The set of subunit curves joining $x$ and $y$ is not empty, by
Chow's theorem, since by (\ref{stratificazione}), the rank of the
Lie algebra generated by $X_1,\dots,X_{m}$ is $n$; hence $d_c$
is a distance on $\G$ inducing the same topology as the standard
Euclidean distance. We shall denote  $B_c(x,r)$
 the open  balls associated with $d_c$.   
The cc-distance is well behaved with respect to left
translations and dilations, that is
\begin{equation}\label{well behaved}
\begin{split}d_c(z\cdot x,z\cdot y)=d_c(x,y)\quad &,\quad
d_c(\delta_\lambda(x),\delta_\lambda(y))=\lambda d_c(x,y)
\end{split}\end{equation}
for
$x,y,z\in\G$ and $\lambda>0$.

We have also
 \begin{equation*}\label{ball measures}
|B_c(x,r)| = r^Q |B_c(0,1)| \quad\mbox{and}\quad |\partial B_c(x,r)|_\G (\G)
= r^{Q-1} |\partial B_c(0,1)|_\G (\G).
\end{equation*}

Denote by $Y$ the vector field $\frac{\partial}{\partial y}$ in $\hat\G:=\G\times\R$. The Lie group $\hat\G$
 is a Carnot group; its Lie algebra $\hat{\mathfrak g}$ admits the
 stratification
 \begin{equation}\label{stratificazione GxR}
\hat{\mathfrak{g}}=\hat V_1\oplus V_2\oplus\cdots\oplus V_\kappa,
\end{equation}
where $\hat V_1 =\mathrm{span}\,\{Y,V_1\}$. Since the  basis $\{X_1,
\dots,X_m\}$ of $V_1$ has been already fixed once and for all, the associated 
basis for $\hat{V_1}$ will be $\{X_1,\dots, X_m,Y\}$.

The following statement follows trivially from the definition of Carnot-Carath\'eodory
distance, keeping into account that the coefficients of $X_1,\cdots,X_m$ in $\hat\G$
are independent of $y$.

\begin{lemma}\label{balls} Denote by $\hat B_c((x,y),r)$ a Carnot-Carath\'eodory ball in
$\hat \G$ centered at the point $(x,y)\in\hat \G$ and $B_c(x,r)$ the Carnot-Carath\'eodory ball in
$\G$ centered at the point $x\in \G$. Then
$$
\hat B_c((x,0),r)\cap \{y=0\} = B_c(x,r)\times\{0\}.
$$
Moreover, if $(x,y) \in K$, where $K\subset \G\times\R$ is a
compact set, and $r\le r_0$ there exist $\sigma_1, \sigma_2>0$ (independent of $r$ and $(x,y)$) such that
$$
\hat B_c((x,y),\sigma_1r) \subset B_c(x,r)\times (y-r,y+r) \subset  \hat B_c((x,y),\sigma_2r).
$$
\end{lemma}

%

\begin{definition}[see \cite{muckenhoupt}, \cite{calderon}] A function $\omega\in L^1_{\mathrm{loc}}(\G)$
is said to be a $A_2$-weight with respect to the cc-metric of $\G$ if
$$
\sup_{x\in\G,\,r>0} \ave_{B_c(x,r)}\omega(y)\,dy \cdot  \ave_{B_c(x,r)}\omega(y)^{-1}\,dy < \infty.
$$

The following remark will be crucial in Section \ref{main}.
\begin{remark}\label{A2 product} By Lemma \ref{balls}, the function $\omega(x,y)=|y|^a$ is
a $A_2$-weight with respect to the cc-metric of $\G\times\R$ if and only if $-1<a<1$.

\end{remark}

\end{definition}
The following result, that is the counterpart in the sub-elliptic framework of the Euclidean setting (see e.g. \cite{FKS} and \cite{sta}), can be found in \cite{lu}.
This idea goes back (at least for the so-called ``Grushin type'' vector fields) to \cite{FL}
and \cite{FS_ASNP}. Basically, this is possible thanks to weighted Sobolev-Poincar\'e inequalities
in Carnot groups.

For further results concerning the boundary Harnack principle in Carnot groups
we refer to \cite{FF_CPDE}.

\begin{theorem}\label{FKS} Let $\G$ be a Carnot group, and let $\Omega\subset\G$ be an open set. Let now
$\omega\in L^1_{\mathrm{loc}}(\G)$ be a $A_2$-weight with respect to the Carnot-Carath\'eodory metric $d_c$ of $\G$.
Then, if $u\in W^{1,2}_\G(\Omega, \omega dx)$ is a  weak solution of
\begin{equation}\label{eq A2}
\mathrm{div }_\G \, (\omega\, \nabla_\G u) = 0,
\end{equation}
then $u$ is locally H\"older continuous in $\Omega$.
If, in addition, $u\ge 0$, then  there exist $C,b>0$ (independent of $u$) such that
the following invariant Harnack inequality holds:
$$
\sup_{B_c(x,r)} u \le C \inf_{B_c(x,r)} u 
$$
for any metric ball $B_c(x,r)$ such that $B_c(x,br)\subset \Omega$. 

Suppose now $\Omega$ satisfies the following local condition (S): for any $x_0\in\partial\Omega$
there exist $r_0>0$ and $\alpha>0$ such that
$$
|B_c(x_0,r)\cap \Omega^c|\ge \alpha |B_c(x_0,r)|\quad\mbox{for $r<r_0$.}
$$
Then $u$ is locally H\"older continuous in $\overline{\Omega}$.

\end{theorem}

\section{Fractional powers of subelliptic Laplacians}\label{fractional}
%
%
%
\begin{definition} Let $\alpha\in\,\C$. We call $K_{\alpha}$ a kernel of type $\alpha$ (according to Folland) a distribution which is smooth away from 0 and 
$\G$-homogeneus of degree $\alpha-Q.$
\end{definition}
\begin{remark}\label{eqnalfa}Let $K_{\alpha}$ be a positive kernel of type ${\alpha}$; 
then there exist  $m,M\,\in\,\R$, with $0<m\leq\,M<\infty$, such that
\begin{equation*}
m\,d(y,0)^{\alpha-Q}<K_{\alpha}(y)<Md(y,0)^{\alpha-Q},
\end{equation*}
for any $y\,\in\,\G.$
\end{remark}


  \begin{proposition}\label{pdalpha} Suppose $0<\beta<Q$. Denote by $h=h(t,x)$ the fundamental
solution of $\mc L+\partial/\partial t$ (see \cite{folland}, Proposition 3.3). Then the integral
$$
R_\beta(x) =\frac{1}{\Gamma(\beta/2)}
\int_0^{\infty}t^{\frac{\beta}{2}-1}h(t,x)\, dt
$$
converges absolutely for $x\neq 0$. In addition, $R_\beta$ is a kernel of type $\beta$.

Moreover
\begin{itemize}
\item [i)] $R_2$ is the fundamental solution of $\mc L$;
\item[ii)] if $\alpha\in (0,2)$ and $u\in\mc D(\G)$, then
$$
\mc L^{\alpha/2} u = \mc Lu\ast R_{2-\alpha}.
$$
\item[iii)] 
the kernels $R_{\alpha}$ admit the following convolution rule: if $\alpha>0,$ $\beta>0$ and $x\neq 0$, then
$$R_{\alpha+\beta}(x)=R_{\alpha}(x)\ast R_{\beta}(x).$$
\end{itemize}
 \end{proposition}

\begin{proof}
These results are basically contained in \cite{folland}. Let us sketch the proof of ii): by \cite{folland}, Theorem 3.15, iii), and Proposition 3.18, keeping in mind that $\mc D(\G)$ is contained in
the domain of all real powers of $\mc L$, we obtain
$$
\mc L^{\alpha/2} u=\mc L^{(\alpha-2)/2}\mc L u = \mc Lu\ast R_{2-\alpha}.
$$
\end{proof}

\begin{remark}\label{Rbeta} If $\beta<0$, $\beta\notin\{0,-2,-4,\cdots\}$, then again
$$
\widetilde R_\beta(x) =\frac{\frac{\beta}{2}}{\Gamma(\beta/2)}
\int_0^{\infty}t^{\frac{\beta}{2} -1}h(t,x)\, dt
$$
defines a smooth function in $\G\setminus\{0\}$, since $t\to h(t,x)$ vanishes of infinite order
as $t\to 0$ if $x\neq 0$. In addition, $\widetilde R_\beta$ is positive and $\G$-homogeneous of
degree $\beta-Q$. However, unlike $R_\beta$ for $0<\beta<Q$,
$\widetilde R_\beta$ is not a kernel of type $\beta$, since it
does not belong to $L^1_{\mathrm{loc}}(\G)$. Integrating by parts, it is easy to see also
that, if $0<\alpha<2$, then
$$
\mc L R_{2-\alpha} = \widetilde R_{-\alpha}
$$
for $x\neq 0$.

\end{remark}
%
%

\begin{definition} 
We set (we remind that $R_\beta>0$ for $0<\beta<Q$)
$$
\rho (x) = R_{2-\alpha}^{1/(2-\alpha-Q)}.
$$
It is easy to see that $\rho$ is an $\G$-homogeneous norm in $\G,$ smooth outside of the
origin. In addition, $d(x,y):=\rho(y^{-1}x)$ is a quasi-distance in $\G$.
In turn, $d$ is equivalent to
the Carnot-Carath\'eodory distance on $G$, as well as to any other $\G$-homogeneous left invariant
distance on $\G$.
\end{definition}

\begin{proposition} Denote by $B_\rho=B_\rho(x,r)$ the metric balls given by $\rho$. We have:
\begin{equation}\label{equiv distances}
m d_c(x,y)\le d(x,y) \le M d_c(x,y)\quad\mbox{for all $x,y\in\G$};
\end{equation}
\begin{equation}\label{volume ball}
m r^Q \le |B_\rho(x,r)| \le M r^Q;
\end{equation}
\begin{equation}\label{perimeter ball}
m r^{Q-1} \le \mc H^{Q-1}_\G(\partial B_\rho(x,r)) \le M  r^{Q-1}.
\end{equation}

\end{proposition}

\begin{definition} We denote by $x\to \wcheck x$ the ``semicheck'' map  
$$(x_1,\dots,x_n)\to
((-1)^{d_1}x_1,(-1)^{d_2}x_2,\dots,(-1)^{d_n}x_n).$$
From now on, we adopt the following notation:$\wcheck f(x,t):=f( \wcheck x,t)$ for any function $f$ 
defined in $\G\times\R$.

\end{definition}

\begin{theorem} \label{wcheck}
We have:
\begin{itemize}
\item[i)] if $j=1,\dots,m$, then $X_j \wcheck = - \wcheck X_j$. In particular, $\mc L  \wcheck =  \wcheck\mc L$;
\item[ii)] if $h$ is the fundamental solution of $\partial_t+\mc L$, then $\wcheck h= h$;
\item[iii)] if $\alpha>0$, $R_\alpha = \wcheck R_\alpha$ and $\tilde R_{-\alpha} = \wcheck \tilde R_{-\alpha}$. In particular, $\wcheck\rho =\rho$;
\item[iv)] $d_c(\wcheck x, \wcheck y)=d_c(x,y)$ for all $x,y\in\G$.
 \item[v)] if $E\subset \G$ is a $\G$-Cacciopoli set, then the perimeter measure $|\partial E|_\G$ is semicheck-invariant.
\end{itemize}
\end{theorem}

\begin{proof} The core of the proof relies in the following identity.  If $p_{jk}$ are the polynomials defined in Proposition \ref{campi omogenei0},
then
\begin{equation}\label{whomo}
p_{j,k}(\wcheck x) = (-1)^{d_k-1} p_{j,k}(x).
\end{equation}
To prove \eqref{whomo}, we remind that $p_{j,k}$ is a $\G$-homogeneous polynomials of degree $d_k-1$.
Let now $\alpha=(\alpha_1,\dots,\alpha_n)$ be a multi-index, and let $x^\alpha$ be an arbitrary 
$\G$-homogeneous monomial of degree $d_k-1$, i.e. assume 
\begin{equation}\label{sum1}
d_1\alpha_1+\cdots + d_n\alpha_n = d_k-1.
\end{equation}

We have but to show that \eqref{whomo} holds for $x^\alpha$.

If $\ell=1,\dots,n$, we set $I_\ell:=\{i\, ;\, d_i=\ell\}$. Gathering in \eqref{sum1} the terms with
$d_i=\ell$, identity \eqref{sum1} becomes
\begin{equation}\label{sum2}
\sum_{\ell} d_\ell ( \sum_{i\in I_\ell}\alpha_i) = d_k-1.
\end{equation}
Then
$$
(\wcheck x)^\alpha = 
(-1)^{\sum_{\ell} d_\ell ( \sum_{i\in I_\ell}\alpha_i)}\, x^\alpha,
$$
and the assertion follows by \eqref{sum2}.

Let us prove now i). If $u$ is a (say) smooth function, by \eqref{whomo}, we have
\begin{equation*}\begin{split}
X_j & (\wcheck u( x))  = X_j (u(\wcheck x))
\\&
=-(\partial_j u)(\wcheck x) + \sum_{d_k>1} p_{j,k}(x)(-1)^{d_k} (\partial_k u)(\wcheck x)
\\&
=-(\partial_j u)(\wcheck x) - \sum_{d_k>1} p_{j,k}(\wcheck x) (\partial_k u)(\wcheck x)
\\&
= - (X_j u)(\wcheck x) = - \wcheck (X_j u)(x).
\end{split}\end{equation*}
In order to prove ii), let us show preliminarily that $h(t, \wcheck x)$ is still a fundamental
solution of $\partial_t +\mc L$. Indeed, if $u\in \mc D(\R\times\G)$, we have
\begin{equation*}\begin{split}
&\Scal{(\partial_t +\mc L) h(t, \wcheck x)}{u(t,x)} 
=
\Scal{h(t, \wcheck x)}{(-\partial_t +\mc L) u(t,x)}
\\&
=
\Scal{h}{\wcheck (-\partial_t +\mc L) u}
=
\Scal{h}{ (-\partial_t +\mc L) \wcheck u}
\\&
=
\Scal{(\partial_t +\mc L) h}{  \wcheck u} = \wcheck u(0,0) = u(0,0).
\end{split}\end{equation*}
Therefore, the function
$$
h_0 := h - h\wcheck
$$
vanishes at $t=0$ and solves $(\partial_t +\mc L)h_0=0$, being 
in particular smooth in $\R\times\G$, by the hypoellipticity of $\partial_t +\mc L$ (\cite{Ho}).
By \cite{folland}, Corollary 3.5, $h_0(t,x)\to 0$ as $x\to \infty$ uniformly for
$t$ in a bounded interval. Thus we can apply the standard ``parabolic'' maximum
principle to conclude that $h_0\equiv 0$, and then ii) follows.

The proof of iii) is straightforward.
To prove iv), it is enough to show that, if $x,y\in\G$ and $\gamma$ is a horizontal curve
joining $x$ and $y$ with sub-Riemannian length $\ell(\gamma)$, then $\wcheck \gamma$ is still horizontal,
$\ell(\wcheck \gamma)= \ell(\gamma)$, and, obviously, joins $x$ and $y$.

By assumption, we can write
$$
\gamma'(t) = \sum_{j=1}^m a_j(t) X_j(\gamma(t)),\quad t\in [0,1],
$$
i.e, if for any $p\in\G$ we write $p_\ell$ for the $\ell$-th component of
$p$ in exponential coordinates, for $\ell=1,2,\dots,n$, then
\begin{equation}\label{gamma_ell1}
\gamma_\ell' = \sum_{j=1}^m a_j(X_j(\gamma))_\ell=
\sum_{j=1}^m a_j (e_j+\sum_{d_k>1}p_{j,k}(\gamma)e_k)_\ell,
\end{equation}
with
$$
\int_0^1\big(\sum_j a_j^2(t)\big)^{1/2}\, dt = \ell(\gamma).
$$
Notice that \eqref{gamma_ell1} reads as follows:
\begin{equation}\label{gamma_ell2}
\gamma_\ell'  =
\left\{
\begin{aligned} & a_\ell & \mbox{if $1\le\ell\le m$}
\\& \sum_{j}a_jp_{j,\ell}(\gamma) & \mbox{if $\ell> m$}.
\end{aligned}
\right.
\end{equation}
Our assertion will follow by showing that
\begin{equation}\label{gamma_ell4}
(\wcheck \gamma)'(t) = - \sum_{j=1}^m a_j(t) X_j(\wcheck\gamma(t)),\quad t\in [0,1],
\end{equation}
Indeed, by \eqref{whomo},
\begin{equation*}\begin{split}X_j & (\wcheck \gamma) = e_j + \sum_{d_k>1} p_{j,k}(\wcheck \gamma)e_k 
\\&
= (-1)^{d_1-1}e_j + \sum_{d_k>1}(-1)^{d_k-1} p_{j,k}( \gamma)e_k ,
\end{split}\end{equation*}
so that, keeping in mind \eqref{gamma_ell2},
\begin{equation*}\begin{split}\label{gamma_ell3}
( & \sum_{j=1}^m a_j X_j  (\wcheck \gamma))_\ell 
=
\left\{
\begin{aligned} & - (-1)^{d_1} a_\ell = - (\wcheck \gamma)_\ell & \mbox{if $1\le\ell\le m$}
\\& - (-1)^{d_\ell}\sum_{j}a_j p_{j,\ell}( \gamma) = - (\wcheck \gamma)_\ell & \mbox{if $\ell> m$}.
\end{aligned}
\right.
\end{split}\end{equation*}
This proves \eqref{gamma_ell4} and achieves the proof of the theorem, since v) is a
straightforward consequence of i).
\end{proof}

\begin{corollary} If $\alpha>0$ and $j=1,\dots,m$, then
$$
\wcheck (X_j R_\alpha) = - X_j R_\alpha
\quad\mbox{and}\quad
\wcheck (X_j \tilde R_{-\alpha}) = - X_j \tilde R_{-\alpha}.
$$
\end{corollary}

We follow the guidelines of \cite{folland}, Section 3. We have:
\begin{theorem} The operator $\mc L$ is a positive self-adjoint operator with domain $W_\G^{2,2}(\G)$. Denote now
by $\{E(\lambda)\}$ the spectral resolution of $\mc L$ in $L^2(\G)$. If $\alpha >0$ then
$$
\mc L^{\alpha/2} = \int_0^{+\infty} \lambda^{\alpha/2} dE(\lambda)
$$
with domain
$$
W^{\alpha,2}_\G(\G):=\{u\in L^2(\G)\; :\; \int_0^{+\infty} \lambda^{\alpha} d\scal{E(\lambda)u}{u}<\infty\},
$$
endowed with the graph norm.
\end{theorem}

\begin{theorem} If  $u\in\mc S(\G)$,
and $0<\alpha<2$, then $\mc L^{\alpha/2} u\in L^2(\G)$, and
\begin{equation*}\begin{split}
\mc L^{\alpha/2} u(x) & = \int_\G \Big(u(xy)-u(x)-\omega(y)\scal{\nabla_\G u(x)}{y}\Big)\widetilde R_{-\alpha}(y)\; dy
\\&= \mathrm{P.V.}  \int_\G (u(y)-u(x))\widetilde R_{-\alpha}(y^{-1}x)\; dy
,
\end{split}\end{equation*}
where $\omega$ is the characteristic function of the unit ball $B_\rho(0,1)$.
\end{theorem}

\begin{proof} First of all, we notice that the map 
$$y\to (u(xy)-u(x)-\omega(y)\scal{\nabla_\G u(x)}{y})\widetilde R_{-\alpha}(y)$$
belongs to $L^1(\G)$.
Indeed, 
$$
(u(xy)-u(x)-\omega(y)\scal{\nabla_\G u(x)}{y})\widetilde R_{-\alpha}(y)=O(\rho(y)^{-Q-\alpha})$$
 as 
$y\to\infty$, and 
$$
(u(xy)-u(x)-\omega(y)\scal{\nabla_\G u(x)}{y})\widetilde R_{-\alpha}(y)=O(\rho(y)^{-Q+2-\alpha})$$
 as 
$y\to 0$, since (\cite{folland_stein}, (1.37))
$$
u(xy)-u(x)-\scal{\nabla_\G u(x)}{y} = O(\rho(y)^2)
$$

If $\eps>0$, keeping in mind that both $\rho$ and $\widetilde R_{-\alpha}$ are check-invariant, we can write
$$
\int_{\rho(y^{-1}x)>\eps} (u(y)-u(x))\widetilde R_{-\alpha}(y^{-1}x)\; dy
= \int_{\rho(y)>\eps} (u(xy)-u(x))\widetilde R_{-\alpha}(y)\; dy.
$$
Notice both  integral are absolutely convergent, since $y\to (u(xy)-u(x))\widetilde R_{-\alpha}(y)$ is a smooth
function away from the origin and $(u(xy)-u(x))\widetilde R_{-\alpha}(y) =O(\rho(y)^{-Q-\alpha})$ as 
$y\to\infty$. On the other hand, the map $y\to\omega(y)\scal{\nabla_\G u(x)}{y}\widetilde R_{-\alpha}(y)$ (that belongs to 
$L^1(\{\rho(y)>\eps\})$) has zero integral, since $\omega(y) \widetilde R_{-\alpha}(y)$ is check-invariant, 
whereas $\scal{\nabla_\G u(x)}{y^{-1}} = \scal{\nabla_\G u(x)}{y}$. Therefore, we can write
\begin{equation*}\begin{split}
\int_{\rho(y^{-1}x)>\eps} & (u(y)-u(x))\widetilde R_{-\alpha}(y^{-1}x)\; dy
\\&
= \int_{\rho(y)>\eps}  \Big(u(xy)-u(x)-\omega(y)\scal{\nabla_\G u(x)}{y}\Big)\widetilde R_{-\alpha}(y)\; dy,
\end{split}\end{equation*}
so that
$$
 \int_\G \Big(u(xy)-u(x)-\omega(y)\scal{\nabla_\G u(x)}{y}\Big)\widetilde R_{-\alpha}(y)\; dy
= \mathrm{P.V.}  \int_\G (u(y)-u(x))\widetilde R_{-\alpha}(y^{-1}x)\; dy.
$$
We want to show now that
\begin{equation}\label{toprove}\begin{split}
\int_{\rho(y^{-1}x)>\eps} & (u(y)-u(x))\widetilde R_{-\alpha}(y^{-1}x)\; dy 
\\&
=
  \int_{\rho(y^{-1}x)>\eps}  \mc Lu(y)R_{2-\alpha}(x^{-1}y)\; dy + o(1)
\end{split}\end{equation}
as $\eps \to 0$. Notice both integrals absolutely converge at infinity.

Take now $R>\eps$. By Green identity (see e.g. \cite{BLU},  formula (5.43b)), we have
\begin{equation*}\begin{split}
\int_{\eps < \rho(y^{-1}x) <R} & (u(y)-u(x))\widetilde R_{-\alpha}(y^{-1}x)\; dy
\\& = 
\int_{\eps < \rho(y^{-1}x) <R}  (u(y)-u(x))\mc L R_{2-\alpha}(y^{-1}x)\; dy
\\& = 
  \int_{\eps < \rho(y^{-1}x) <R}  \mc Lu(y)R_{2-\alpha}(x^{-1}y)\; dy
  \\& +
   \int_{\eps = \rho(y^{-1}x)}  R_{2-\alpha}(x^{-1}y) \sum_j X_j  (u(y)-u(x)) \scal{X_j}{\nu}\; d\mc H^{n-1}(y)
    \\& -
   \int_{\eps = \rho(y^{-1}x)}  (u(y)-u(x)) \sum_j X_j R_{2-\alpha}(x^{-1}y)\scal{X_j}{\nu}\; d\mc H^{n-1}(y)
    \\& +
   \int_{R = \rho(y^{-1}x)}  R_{2-\alpha}(x^{-1}y) \sum_j X_j  (u(y)-u(x)) \scal{X_j}{\nu}\; d\mc H^{n-1}(y)
    \\& -
   \int_{R = \rho(y^{-1}x)}  (u(y)-u(x)) \sum_j X_j R_{2-\alpha}(x^{-1}y)\scal{X_j}{\nu}\; d\mc H^{n-1}(y)
   \\& = 
  \int_{\eps < \rho(y^{-1}x) <R}  \mc Lu(y)R_{2-\alpha}(x^{-1}y)\; dy
  \\&
  + I^1(\eps)+ I^2(\eps) + J^1(R)+ J^2(R),
  \end{split}\end{equation*}
where $\nu$ in the outward unit normal to $\{\eps < \rho(y^{-1}x) <R\}$. Obviously, $J_1$
vanishes 
as $R\to\infty$. Again, by Remark \ref{eqnalfa}, if $R$ is large, we have
\begin{equation*}\begin{split}
| & J^2( R ) | \le C|u(x)| 
R^{1-\alpha-Q} \int_{R = \rho(y^{-1}x)}   \sum_j |\scal{X_j}{\nu}|\; d\mc H^{n-1}(y)
\\&\le
C|u(x)| 
R^{1-\alpha-Q} \int_{R = \rho(y^{-1}x)}   \; d\mc H^{Q-1}_\G (y)\quad\mbox{(by 
Proposition \ref{perimetro regolare})}
\\&
= O(R^{-\alpha}),
\end{split}\end{equation*}
 by \eqref{perimeter ball}. Thus we can take above the limit as $R\to\infty$ and we get
\begin{equation*}\begin{split}
\int_{\eps < \rho(y^{-1}x) } & (u(y)-u(x))\widetilde R_{-\alpha}(y^{-1}x)\; dy
\\& = 
  \int_{\eps < \rho(y^{-1}x) }  \mc Lu(y)R_{2-\alpha}(x^{-1}y)\; dy
  + I^1(\eps)+ I^2(\eps) 
  \end{split}\end{equation*}
(notice again both integrals are absolutely convergent).

Thus, \eqref{toprove} will follow by showing that $ I^1(\eps)+ I^2(\eps) =
o(1)$ as $\eps\to 0$.

Consider $I_1(\eps)$. First of all, we notice that
\begin{equation}\label{eq10}\begin{split}
  \int_{\eps  = \rho(y^{-1}x)} & R_{2-\alpha}(x^{-1}y)  \scal{X_j}{\nu}\; d\mc H^{n-1}(y)
  \\&
  =  \int_{\eps  = \rho(y)}  R_{2-\alpha}(y)  \scal{X_j}{\nu}\; d\mc H^{n-1}(y)
   = 0
\end{split}\end{equation}
for $j=1,\dots, m$. Indeed we can write
\begin{equation*}\begin{split}
  =  \int_{\eps  = \rho} & R_{2-\alpha}  \scal{X_j}{\nu}\; d\mc H^{n-1}
  \\&
   =  \int_{\eps  = \rho}  R_{2-\alpha} (X_jR_{2-\alpha})|\nabla R_{2-\alpha}|^{-1}\; d\mc H^{n-1},
\end{split}\end{equation*}
and \eqref{eq10} follows, since $\rho$, $R_{2-\alpha}$, $|\nabla R_{2-\alpha}|$, and $\mc H^{n-1}$
are even under the change of variables $y\to \wcheck y$, whereas $X_jR_{2-\alpha}$ is
odd. Thus, by Proposition \ref{perimetro regolare}, we can write
\begin{equation*}\begin{split}
I_1& (\eps) =  \int_{\eps = \rho(y)}  R_{2-\alpha}(y) \sum_j (X_j  u)(xy) \scal{X_j}{\nu}\; d\mc H^{n-1}(y)
\\& =  \int_{\eps = \rho(y)}  R_{2-\alpha}(y) \sum_j [ X_j  u(xy)- X_ju(x)] \scal{X_j}{\nu}\; d\mc H^{n-1}(y)
\\&
\le
C\max |X^2u|\,\eps^{3-\alpha-Q} \int_{\eps = \rho} \; d\mc H^{Q-1}_\G (y)
\\&
= O(\eps^{2-\alpha}) = o(1)\quad\mbox{as $\eps\to 0$}.
\end{split}\end{equation*}
Finally, $I^2(\eps)$ can be estimated by similar arguments. We write
$$
I^2(\eps)=- \int_{\eps = \rho}  (u(xy)-u(x)) \sum_j (X_j R_{2-\alpha})\scal{X_j}{\nu}\; d\mc H^{n-1}(y),
$$
and we notice that, if $1\le  \ell\le m$
$$
\int_{\eps = \rho} y_\ell \, \sum_j (X_j R_{2-\alpha})\scal{X_j}{\nu}\; d\mc H^{n-1}(y)=0.
$$
Indeed, keeping again in mind  Proposition \ref{perimetro regolare},  we have
\begin{equation*}\begin{split}
\int_{\eps = \rho} &  y_\ell \, \sum_j (X_j R_{2-\alpha})\scal{X_j}{\nu}\; d\mc H^{n-1}(y)
\\&
=
\int_{\eps = \rho}   y_\ell \,  \sum_j |X_j R_{2-\alpha}|^2 |\nabla R_{2-\alpha}|^{-1} \; d\mc H^{n-1}(y)
\\&
=
\int_{\eps = \rho}   y_\ell \,   |\nabla_\G R_{2-\alpha}|(\sum_j\scal{X_j}{\nu}^2)^{1/2} \; d\mc H^{n-1}(y)
\\&
=
\int_{\eps = \rho}   y_\ell \,   |\nabla_\G R_{2-\alpha}| \; d\mc H^{Q-1}_\G (y) = 0
\end{split}\end{equation*}
since both $ |\nabla_\G R_{2-\alpha}| $ and $\mc H^{Q-1}_\G$ are even with respect to the
change of variable $y\to \wcheck w$, whereas $y_\ell$ is odd.

Therefore, keeping in mind Taylor inequality in $\G$ (see, e.g. \cite{folland_stein} Theorem 1.37),  
as well as Remark \ref{eqnalfa} and, again, Proposition \ref{perimetro regolare}, we can write
\begin{equation*}\begin{split}
| & I^2(\eps)| 
\\&
 = \Big| \int_{\eps = \rho}  (u(xy)-u(x)-\sum_\ell X_\ell u(x)y_\ell)
 \sum_j (X_j R_{2-\alpha})\scal{X_j}{\nu}\; d\mc H^{n-1}(y) \Big|
 \\&
\hphantom{xxx} =
 O(\eps^{3-\alpha-Q}) \mc H^{Q-1}_\G \big(\{\eps = \rho\}\big)
 = O(\eps^{2-\alpha}) = o(1).
 \end{split}\end{equation*}
This achieves the proof of \eqref{toprove}. Taking the limit as $\eps\to 0$ in \eqref{toprove},
and keeping in mind that $R_{2-\alpha} \mc Lu \in L^1(\G)$,
we get eventually
\begin{equation*}\begin{split}
\mathrm{P.V} & \int_{\G}  (u(y)-u(x))\widetilde R_{-\alpha}(y^{-1}x)\; dy 
=
  \int_{\G}  \mc Lu(y)R_{2-\alpha}(x^{-1}y)\; dy
  \\&
  =
  \mc L^{\alpha/2} u,
\end{split}\end{equation*}
by Proposition \ref{pdalpha}.
This achieves the proof of the theorem.
\end{proof}

\section{Main results}\label{main}

\begin{proposition}[see also Caffarelli \& Silvestre \cite{caffarelli_silvestre}]
\label{eqdiff} If $-\infty<\alpha<1$, the boundary value problem
\begin{equation}\label{equazionediff}
\left\{\begin{array}{l}
-t^{\alpha}\phi''+\phi=0\\
\phi(0)=1\\
\lim_{t\to+\infty}\phi(t)=0
\end{array}
\right.
\end{equation}
has a solution $\phi\in \mathbf C^{2-\alpha}([0,\infty))$ 
of the form
$$
\phi(t) = c_\alpha\, t^{1/2} K_{1/2k}(k^{-1}t^{k}),
$$
where $c_\alpha:= 2^{1-1/2k}\Gamma(1/2k)^{-1}k^{-1/2k}>0$ is a positive constant, $k= \frac{2-\alpha}{2}$, and $K_{1/2k}$ is the modified Bessel function of second kind
(see \cite{watson}).  We know that 
\begin{itemize}
\item[i)] $0<\phi <1$. Moreover $\phi'(t)$ has a finite limit as $t\to 0$ and, recursively, 
$$
t^{\alpha+h-2}\phi^{(h)}(t)\mbox{ has a finite limit as $t\to 0$}
$$
for $h=2,3,\dots$;
\item[ii)] $\phi'\in L^2((0,\infty))$;
\item[iii)] $\phi(t)= c\, \sqrt{\frac{\pi k}{2}}\,t^{\alpha/2}\,e^{-t^k/k}\big(1+O(\frac1t)\big)$ as $t\to\infty$;
\item[iv)] $\phi^{(h)}(t) = c_h \,t^{\alpha (1-h)/2}\,e^{-t^k/k}\big(1+o(1)\big)$ as $t\to\infty$ for $h=1,2,\dots$.
\end{itemize}
\end{proposition}

\begin{proof}
By iteration, we can reduce ourselves to prove the assertion for $h=1$. Since $\phi$ is convex, $\phi'(t)\to 0$
as $t\to\infty$ and we can write
$$
\phi'(t) = \int_t^\infty s^{-\alpha}\phi(s)\, ds = c\, \sqrt{\frac{\pi k}{2}}\, \int_t^\infty s^{-\alpha/2} \,e^{-s^k/k}\big(1+o(1)\big)\, ds.
$$
Then the estimate follows by the de l'H\^opital's rule.
\end{proof}

\begin{remark} The exact value of $\phi'(0)$ can be 
 explicitly computed keeping in mind that
 $$
 \phi'(0) = c_\alpha \int_0^\infty s^{-\alpha + 1/2} K_{1/2k}(\frac1k s^k)\, ds
 = \frac{c_\alpha}{k} \int_0^\infty t^{(a+1)/2} K_{1/2k}(\frac1k t)\, dt,
 $$
and that the last integral in turn can be explicitly evaluated by \cite{GR}, 6.561 (16).
\end{remark}

Put $\theta:= (1-a)^{a-1}$. If $u\in W^{1-a,2}(\G)$, for $y>0$ we set
\begin{equation}\label{def v}\begin{split}
v(\cdot,y) : &= \phi(\theta y^{1-a} \, \mc L^{(1-a)/2})u := \int_0^\infty \phi(\theta y^{1-a} \, \lambda^{(1-a)/2})\,dE(\lambda)u,
\end{split}\end{equation}

Notice $v$ is well defined since $\phi$ is continuous and bounded in $[0,\infty)$.

Choose now
$$
\alpha =-\frac{2a}{1-a}.
$$
\begin{proposition}\label{prop v}Set 
$ \Sigma_+=\G_x\times (0,1)_y$ and $ \Sigma_+^\eps=\G_x\times (\eps,1)_y$.
If $$s\ge 1 - \frac{a+1}{2}\quad\mbox{and}\quad
u\in W^{s,2}(\G),$$
 then $v\in W^{1,2}_{\hat \G}(\Sigma_+; \, y^adx\,dy)$ and
 \begin{equation}\label{prop v eq:1}
 \|v\|_{W^{1,2}_{\hat G}(\Sigma_+; \, y^adx\,dy)} \le C\,  \|u\|_{W^{s,2}(\G)} .
\end{equation}

Moreover, if 
$$s\ge 2 - \frac{a+1}{2}\quad\mbox{and}\quad
u\in W^{s,2}(\G),$$
 then $v\in W^{2,2}_{\hat G}(\Sigma_+^\eps; \, y^adx\,dy)$
 for any $\eps>0$.

\end{proposition}

\begin{proof} 

The function $v$ belongs to $L^2(\ \Sigma_+; \, y^adx\,dy)$. Indeed
\begin{equation*}\begin{split}
\| v & \|_{L^2(\Sigma_+; \, y^adx\,dy)}^2 
= \int_0^1 dy\,y^a\| v(y,\cdot)\|_{L^2(\mathbb R^{n})}^2
\\&
= \int_0^1 dy\,  y^a \int_0^\infty \phi^2(\theta y^{1-a} \, \lambda^{(1-a)/2})\,d\| E(\lambda)u\|^2
\le C\|u\|_{L^2(\G)}^2,
\end{split}\end{equation*}
since $\phi$ is bounded.

On the other hand, if $\eps\ge 0$,
\begin{equation*}\begin{split}
\| v & \|_{W^{k,2}_{\hat G}(\Sigma_+; \, y^adx\,dy)}^2 
= \sum_{0\le h\le k}\sum_{|\beta|\le k-h} \| \partial_y^kX^\beta v  \|_{L^2(\Sigma_+; \, y^adx\,dy)}^2 
\\&=
\sum_{0\le h\le k}\int_\eps^1 dy\,y^a \int_\G dx\,\sum_{|\beta|\le k-h} \| X^\beta\partial_y^hv\|^2
\\&=
\sum_{0\le h\le k}\int_\eps^1 dy\,y^a \|\partial_y^hv\|^2_{W^{k-h,2}(\G)}
\\&\approx
\sum_{0\le h\le k}\int_\eps^1 dy\,y^a  \int_\G dx \|\mc L^{(k-h)/2}\partial_y^hv\|^2_{W^{k-h,2}(\G)}
\\&=
\sum_{0\le h\le k}\int_\eps^1 dy\,y^a  \int_0^\infty\lambda^{k-h} 
|\partial_y^h \phi(\theta y^{1-a} \, \lambda^{(1-a)/2})|^2\, d\| E(\lambda)u\|^2.
\end{split}\end{equation*}
Recalling that $$\sum_{j=1}^hm_j=m$$
and
$$
\sum_{j=1}^hjm_j=h
$$ 
the last term can be estimated by a sum of terms
of the form
\begin{equation*}\begin{split}
\int_\eps^1 dy\,y^a  \int_0^\infty\lambda^{k-h+(1-a)m} 
y^{2m(1-a)-2h}
| \phi^{(h)}(\theta y^{1-a} \, \lambda^{(1-a)/2})|^2\, d\| E(\lambda)u\|^2,
\end{split}\end{equation*}
with $m\le h$. If we put $y\sqrt{\lambda}=\tau$, the last term is estimated by
\begin{equation}\label{5 marzo}\begin{split}
\int_0^\infty\, & d\| E(\lambda)u\|^2 \lambda^{-\frac{a}{2}+k-h+(1-a)m-(1-a)m+h-\frac{1}{2}}
\\&
\hphantom{xxx}\cdot
\int_{\eps\sqrt{\lambda}}^\infty \tau^{2m(1-a)-2h+2a}
| \phi^{(h)}(\theta \tau^{1-a} )|^2 d(\tau^{1-a})
\\&
=
\int_0^\infty\,  d\| E(\lambda)u\|^2 \lambda^{k-\frac{a+1}{2}}
\\&
\hphantom{xxx}\cdot
\int_{(\eps\sqrt{\lambda})^{1-a}}^\infty s^{2m-2(h-a)/(1-a)}
| \phi^{(h)}(\theta s)|^2 ds
\end{split}\end{equation}
Consider now the case $k=1$ (and therefore $h=m=1$, since the case $h=0$ yields
the $L^2$-estimate we have already proved). Then we can take $\eps=0$ and
the last term becomes
\begin{equation*}\begin{split}
\int_0^\infty\, &  d\| E(\lambda)u\|^2 \lambda^{1-\frac{a+1}{2}}
\cdot
\int_{0}^\infty 
| \phi'(\theta s)|^2 ds
\\&
\le \int_0^\infty\,   d\| E(\lambda)u\|^2 (1+\lambda^{s})
\cdot
\int_{0}^\infty 
| \phi'(\theta s)|^2 ds
\le C \,\|u\|_{W^{s,2}(\G)}^2,
\end{split}\end{equation*}
by ii) above. 

Consider now the case $k=2$. In this case, we take $\eps>0$ and we split the last integral
in \eqref{5 marzo} as
\begin{equation*}\begin{split}
\int_0^1\,   d\| E(\lambda)u\|^2 \cdots 
+ \int_1^\infty\,   d\| E(\lambda)u\|^2  \cdots := I_1 + I_2.
\end{split}\end{equation*}
Obviously,
$$
I_2\le 
\int_1^\infty\,   d\| E(\lambda)u\|^2\,
\int_{\eps^{1-a}}^\infty s^{2m-2(h-a)/(1-a)}
| \phi^{(h)}(\theta s)|^2 ds < \infty,
$$
since $\phi^{(h)}(s)$ vanishes exponentially as $s\to\infty$.
Analogously,
\begin{equation}\label{5 marzo:2}\begin{split}
I_1\le 
\int_0^1\,  d\| E & (\lambda)u\|^2   \lambda^{2-\frac{a+1}{2}}
\int_{(\eps\sqrt{\lambda})^{1-a}}^1 \cdots  ds
\\&
+
\int_0^1\,  d\| E(\lambda)u\|^2 \lambda^{2-\frac{a+1}{2}}
\int_{1}^\infty \cdots  ds
\end{split}\end{equation}
Clearly, the second term in \eqref{5 marzo:2} is finite, again since
since $\phi^{(h)}(s)$ vanishes exponentially as $s\to\infty$.
Thus, we are reduced to estimate
\begin{equation*}\begin{split}
\int_0^1\, & d\| E(\lambda)u\|^2 \lambda^{2-\frac{a+1}{2}}
\\&
\hphantom{xxx}\cdot
\int_{(\eps\sqrt{\lambda})^{1-a}}^1 s^{2m-2(h-a)/(1-a)-2\alpha-2h+4}
|s^{\alpha+h-2} \phi^{(h)}(\theta s)|^2 ds
\\
\le C\,
\int_0^1\, & d\| E(\lambda)u\|^2 \lambda^{2-\frac{a+1}{2}}
\\&
\hphantom{xxx}\cdot
\int_{(\eps\sqrt{\lambda})^{1-a}}^1 s^{2m-2(h-a)/(1-a)-2\alpha-2h+4}
 ds,
\end{split}\end{equation*}
by Proposition \ref{eqdiff}, i).
If we keep in mind that  
$$
\int_0^1\,  d\| E(\lambda)u\|^2 \lambda^{2-\frac{a+1}{2}} < \infty
$$
since $u\in W^{s,2}_\G(\G)$, with $s\ge 2-\frac{a+1}{2}$, to achieve the proof of the proposition
we have but to show that
\begin{equation*}\begin{split}
2- & \frac{a+1}{2} + (1-a)\big( 2m-2(h-a)/(1-a)-2\alpha-2h+5\big)
\\&
= (m-h)(1-a)-h+4 \ge (m-h)(1-a) + 2
 > 0.
\end{split}\end{equation*}
On the other hand, if $h=1$, then necessarily $m=1$, so that $(m-h)(1-a) + 2 =2$,
whereas, if $h=2$, then either $m=1$ or $m=2$. In the first case 
$(m-h)(1-a) + 2 = a + 1 >0$. Finally, if $m=2$, then $(m-h)(1-a) + 2 = 2$,
achieving the proof of the proposition.
\end{proof}

\begin{theorem}[generalized subordination identity]\label{poisson}  If $u\in L^2(\G)$ and $y>0$, we set
\begin{equation*}\begin{split}
v(\cdot,y) : &= \phi(\theta y^{1-a} \, \mc L^{(1-a)/2})u := \int_0^\infty \phi(\theta y^{1-a} \, \lambda^{(1-a)/2})\,dE(\lambda)u,
\end{split}\end{equation*}
where $\theta:= (1-a)^{a-1}$ (we remind that $\phi$ is bounded, and therefore $v\in L^2(\G)$ for $y>0$). 

We denote by $h(t,\cdot)$ the heat kernel associated with $-\mc L$ as in \cite{folland}, and by
$P_\G(\cdot,y)$ the ``Poisson kernel''
\begin{equation}\label{representation of P}
P_\G(\cdot,y) : = C_a\, y^{1-a} \int_0^\infty t^{(a-3)/2} e^{- \frac{y^2}{4t}}  h(t,\cdot)\; dt,
\end{equation}
where
$$C_a = \dfrac{2^{a-1}}{\Gamma((1-a)/2)}.
$$
Then
$$
P_\G(\cdot,y) \ge 0
$$
by \cite{hunt}, \cite{folland}, Theorem 3.1, and
\begin{equation}\label{representation of v}
v(\cdot, y) = u\ast P_\G(\cdot, y).
\end{equation}

\end{theorem}

\begin{proof} By identity (8), p. 182 of \cite{watson}, if $\nu>0$ and $z>0$, we can write
$$
K_\nu(z) = \frac12 \int_0^\infty \xi^{-\nu-1}e^{-\frac12 z(\xi+\frac1\xi)}\;d\xi.
$$
Then (keeping also in mind the definition of $\theta$)
\begin{equation*}\begin{split}
\phi (\theta z) &  :=  \frac12 c_\alpha \theta^{1/2} z^{1/2}
 \int_0^\infty \xi^{-\frac{1}{2k}-1}e^{-\frac{\theta^k}{2k} z^k(\xi+\frac1\xi)}\;d\xi
 \\&
 =
 \frac12 c_\alpha \theta^{1/2} z^{1/2}
 \int_0^\infty \xi^{-\frac{1}{2k}-1}e^{-\frac{\theta^k}{2k} z^k\xi}\;
 e^{-\frac{\theta^k}{2k}  \frac{z^k}\xi}\;d\xi
 \\&
 =
   2^{(a-3)/2} c_\alpha z^{1/2}  \theta^{1/2} 
 \int_0^\infty \tau^{(a-3)/2}e^{- \tau\,z^k}\;
 e^{- \frac{z^k}{4\tau}}\;d\tau
 \quad \mbox{(putting $\frac{\theta^k}{2k}\xi =\tau$)}.
\end{split}\end{equation*}
Hence
\begin{equation*}\begin{split}
\phi(\theta \lambda^{\frac{1-a}{2}}y^{1-a}) &=
 2^{(a-3)/2} c_\alpha \lambda^{\frac{1-a}{4}}y^{(1-a)/2}  \theta^{1/2} 
 \int_0^\infty \tau^{(a-3)/2}e^{- \tau\,\sqrt{\lambda}y}\;
 e^{- \frac{\sqrt{\lambda}y}{4\tau}}\;d\tau
 \\&
 =
 2^{(a-3)/2} c_\alpha \theta^{1/2} y^{1-a}  
 \int_0^\infty t^{(a-3)/2}e^{- \lambda \, t}\;
 e^{- \frac{y^2}{4t}}\;dt,
\end{split}\end{equation*}
putting $y\tau=\sqrt{\lambda} t$. In other words, $\lambda\to \phi(\theta \lambda^{\frac{1-a}{2}}y^{1-a}) $
is, up to a multiplicative constant,  the Laplace transform of
$t\to t^{(a-3)/2}\; e^{- \frac{y^2}{4t}}$.

For sake of brevity we set $C_a:=  2^{(a-3)/2} c_\alpha \theta^{1/2}$ (we remind
that $\alpha$ depends on $a$). Thus  we can write now
\begin{equation*}\begin{split}
v(\cdot,y) & =  C_a y^{1-a}\int_0^\infty  \big( \int_0^\infty t^{(a-3)/2}e^{- \lambda \, t}\;
 e^{- \frac{y^2}{4t}}\;dt  \big)  dE(\lambda)u
 \\&
 =C_a  y^{1-a}\int_0^\infty t^{(a-3)/2} e^{- \frac{y^2}{4t}}
 \big( \int_0^\infty  e^{- \lambda \, t} dE(\lambda)u\big)\; dt
  \\&
=
 C_a y^{1-a} \int_0^\infty t^{(a-3)/2} e^{- \frac{y^2}{4t}} u\ast h(t,\cdot) \; dt
=
 u\ast P_\G(\cdot, y).
 \end{split}\end{equation*}
\end{proof}

\begin{remark}\label{different} Formulas \eqref{representation of v} 
and \eqref{representation of P} make possible
to give a different and more explicit representation of the lifting $v$
of $u$. On the other hand, the estimates of $h(t,\cdot)$ proved in
\cite{folland} and \cite{folland_stein} yield analogous estimates 
for $P_\G$. Indeed, if $I$ is a multi-index, then, if $\rho:=\rho(x)$,
\begin{equation*}\begin{split}
|X^IP_\G(x,y)|& \le C  \int_0^\infty t^{(a-3)/2}|X^Ih(t,x)|\,dt
\\&
= C \rho^{a-1}  \int_0^\infty \tau^{(a-3)/2}|X^Ih({\tau\rho^2},x)|\,d\tau
\end{split}\end{equation*}
By \cite{folland_stein}, identity (1.73), we write now
$$
X^Ih({\tau\rho^2} ,x) = (\sqrt{\tau}\rho)^{-Q-d(I)} |X^Ih(1,\frac{x}{\sqrt{\tau}\rho})|,
$$
and we notice that, since $h(1,\cdot)\in \mc S(\G)$ (by \cite{folland_stein}, 
Proposition 1.74), if $N>0$, then
$$
 |X^Ih(1, \frac{x}{\sqrt{\tau}\rho})| \le C(1+\frac{1}{\sqrt{\tau}})^{-N}
 \le C\frac{\tau^{N/2}}{1+ \tau^{N/2}}.
$$
Thus, eventually, 
\begin{equation*}\begin{split}
|X^IP_\G(x,y)|& \le C \rho^{a-1-Q-d(I)} \int_0^\infty \tau^{(a-3-Q-d(I))/2}\frac{\tau^{N/2}}{1+ \tau^{N/2}}.\,d\tau
\\&
\le
 C \rho^{a-1-Q-d(I)}
\end{split}\end{equation*}
for large $\rho$. Then the lifting convolution $u\ast P_\G$ is well defined as long as $u(x)$ does not
grow too fast as $x\to\infty$. We refer to \cite{caffarelli_silvestre} for similar growth
conditions in the Euclidean setting.

 Moreover, if $u$ is sufficiently smooth,
\begin{equation}
\begin{split}
 &y^a\frac{v(x,y)-v(x,0)}{y}= y^a\frac{u\ast P_\G(\cdot, y)-u(x)}{y}\\
 &=\left(C_a  \int_0^\infty t^{(a-3)/2} e^{- \frac{y^2}{4t}} u\ast h(t,\cdot) \; dt\right.\\
 &\hphantom{aaa}\left.-C_a u(x) \int_{\G}\int_0^\infty t^{(a-3)/2} e^{- \frac{y^2}{4t}}  h(t,\xi^{-1}x) \; dtd\xi\right)\\
 &=C_a \int_{\G}\int_0^\infty t^{(a-3)/2} e^{- \frac{y^2}{4t}}  h(t,\xi^{-1}x) \; dt(u(\xi)-u(x))d\xi
\end{split}
\end{equation}
On the other hand
\begin{equation}
\begin{split}
\lim_{y\to 0^+}C_a\int_0^\infty t^{(a-3)/2} e^{- \frac{y^2}{4t}}  h(t,\xi^{-1}x) \; dt=\tilde{C}_a\tilde{R}_{a-1}.
\end{split}
\end{equation}
Thus
$$
\lim_{y\to 0^+}y^a\frac{v(x,y)-v(x,0)}{y}=C_a\int_{\G}(u(\xi)-u(x))\tilde{R}_{a-1}(\xi)d\xi=\tilde{C}_a\mathcal{L}^{\frac{1-a}{2}}u(x).
$$
\end{remark}
\begin{theorem}\label{premain}
Let $u\in W^{1-a,2}_\G(\G)$ be given, $u\ge 0$, and assume $\mc L^{(1-a)/2}u = 0$
in an open set $\Omega$.
With the notations of Theorem \ref{poisson},
we denote by $\hat v$ the function on $\hat \G$ obtained continuing $v$ by parity
across $y=0$.
Then
\begin{itemize}
\item[i)] $\hat v\ge 0$;
\item[ii)] $\hat v\in W^{1,2}_{\hat\G,  \mathrm{loc}}(\hat\Omega; \, y^adx\,dy)$,
where $\hat\Omega:= \Omega\times (-1,1)$;
\item[iii)] $\hat v$ is a weak solution of the equation
$$
\divgh \big(|y|^a\nabla_{\hat\G} v\big)= 0 \quad\mbox{in $\hat\Omega$.}
$$ 
\end{itemize}
\end{theorem}

\begin{proof} Statement i) follows from previous Theorem \ref{poisson}.
%
%
%

The proofs of ii) and
  iii) are divided in several steps.

\noindent\textbf{Step 1.} From now on, we write $\Sigma_- := \G\times (-1,0)$ and
$\Sigma_-^\eps:= \G\times(-1,-\eps)$.
 If $\eta>0$, we set 
$$u_\eta:= (1+\eta \mc L)^{-1} u:=\int_0^\infty (1+\eta\lambda)^{-1} dE(\lambda)u.
$$
Then $u_\eta\in W^{3-a,2}_\G(\G)$ so that, with the notation of \eqref{def v}, by
Proposition \ref{prop v}, $\hat v_\eta \in W^{2,2}_{\hat \G, \mathrm{loc}}(\Sigma_\pm; \, y^adx\,dy)$.
Moreover, just performing computations, we see that
$$
\divgh \big(|y|^a\nabla_{\hat\G} \hat v_\eta\big)= 0
$$
in $\Sigma_\pm$. 
Moreover, if $\psi\in \mc D(\Sigma_\pm)$, then
$$
\int_{\Sigma_\pm^\eps} \scal{\nabla_{\hat\G} \hat v_\eta}{\nabla_{\hat\G} \psi} | y|^adx\,dy
= 0.
$$

\noindent\textbf{Step 2.} The function $\hat v $ belongs to both 
$W^{1,2}_{\hat \G}(\Sigma_\pm; \, y^adx\,dy)$ (by Proposition
\ref{prop v}) and in addition
$$
\int_{\Sigma_\pm^\eps} \scal{\nabla_{ \G} \hat v}{\nabla_{\hat\G} \psi} | y|^adx\,dy
= 0
$$
for any  $\psi \in\mc D(\hat \Sigma_\pm)$. Indeed, by \eqref{prop v eq:1}, we have
but to notice that $u_\eta \to u$ in $W^{1-a,2}_\G(\G)$. Indeed
$$
\|u_\eta - u\|_{W^{1-a,2}_\G(\G)}^2 =
\int_0^\infty \lambda^{1-a} | (1+\eta\lambda)^{-1} -1 |^2\,d\|E(\lambda)u\|^2 \to 0
$$
as $\eta\to 0$, by dominated convergence theorem. Since the function $y\to |y|^a$ is smooth away from $\{y=0\}$, then
$\hat v_\eta$ is smooth in $\Sigma_\pm$, by classical H\"ormander's theorem (\cite{Ho}).

We notice that this argument going through regularization, equation in non-divergence form, integration
by parts and variational equation is required by our abstract arguments that hides the divergence
structure of the equation. 

\noindent\textbf{Step 3.} Because of the properties of $A_2$-weights, 
$\hat v \in W^{1,1}_{\hat \G,\mathrm{loc}}(\Sigma_\pm)\cap L^1_{\mathrm{loc}}(\Sigma)$. Moreover,
with an obvious meaning of symbols,
\begin{equation}\label{by parts eq: 3}
X_j \hat v  = \widehat{X_jv} \quad\mbox{in $\Sigma$, for $j=1,\dots,m$}
\end{equation}
and
\begin{equation}\label{by parts eq: 4}
\partial_y \hat v = \pm \widehat{\partial_y v}\quad\mbox{in $\Sigma_\pm$}.
\end{equation}
Clearly, this yields $\hat v \in W^{1,2}_{\hat \G,\mathrm{loc}}(\Sigma; \, y^adx\,dy)$
and therefore ii) holds.
Now, \eqref{by parts eq: 3} is obvious. As for \eqref{by parts eq: 4}, if
$\psi \in\mc D(\hat \Omega)$, by divergence theorem
\begin{equation}\label{by parts eq: 5}\begin{split}
\int_\Sigma & \hat v( \partial_y\psi)\, dxdy 
=
 \lim_{\eps\to 0} \int_{\Sigma_+^\eps} \hat v (\partial_y\psi) \, dxdy
 +  \lim_{\eps\to 0} \int_{\Sigma_-^\eps} \hat v (\partial_y\psi) \, dxdy
 \\&
 =  
 \lim_{\eps\to 0} \int_{\Omega}  v(\cdot,\eps) \psi(\cdot,\eps)\, dx
 -  \lim_{\eps\to 0} \int_{\Omega}v(\cdot, -\eps) \psi(\cdot,\eps) \, dx
 \\&
- \lim_{\eps\to 0} \int_{\Sigma_+^\eps} ( \partial_y v) \psi\, dxdy
 +  \lim_{\eps\to 0} \int_{\Sigma_-^\eps} (\widehat{\partial_y v})\psi\, dxdy.
\end{split} \end{equation}
Since $\hat v$ is locally H\"older continuous up to $y=0$
$$
 \lim_{\eps\to 0} \int_{\Omega}  v(\cdot,\eps) \psi(\cdot,\eps)\, dx
 -  \lim_{\eps\to 0} \int_{\Omega}v(\cdot, -\eps) \psi(\cdot,\eps) \, dx
 = 0
 $$
 and the assertion follows.

\noindent\textbf{Step 4.} By divergence theorem, if $\eps\in (0,1)$ 
and $\psi \in\mc D(\hat \Omega)$, then
\begin{equation}\label{by parts eq: 1}\begin{split}
\int_{\Sigma_\pm^\eps} \scal{\nabla_{\hat\G} \hat v}{\nabla_{\hat\G} \psi} | y|^adx\,dy =  
\int_{\Omega}\eps^a \partial_y\hat v (x,\pm\eps)\psi(x,\pm\eps)\,dx
\end{split} \end{equation}
Take now the limit as $\eps\to 0$. Clearly
$$
\int_{\Sigma_\pm^\eps} \scal{\nabla_{\hat\G} \hat v}{\nabla_{\hat\G} \psi} | y|^adx\,dy
\to
\int_{\hat\Omega} \scal{\nabla_{\hat\G} \hat v}{\nabla_{\hat\G} \psi} | y|^adx\,dy
$$
as $\eps\to 0$. If we show that
\begin{equation}\label{by parts eq: 6}
\eps^{a}\partial_y\hat v (x,\pm\eps) \to (1-a)^a \phi'(0) \mc L ^{\frac{1-a}{2}}u \quad \mbox{in $L^2(\G)$},
\end{equation}
then  assertion iii) follows since $\mc L ^{\frac{1-a}{2}}u$ vanishes on $\mathrm{supp}\, \psi$.

To prove \eqref{by parts eq: 6}, we write
\begin{equation*}\begin{split}
\| \eps^{a} & \partial_y\hat v (x,\pm\eps) -  \phi'(0) \mc L ^{\frac{1-a}{2}}u\|^2_{L^2(\G)} 
\\&
=
(1-a)^{2a} \int_0^\infty |\phi'(\theta \lambda^{\frac{1-a}{2}}\eps^{1-a}) - \phi'(0)|^2 \lambda^{1-a}
d\|E(\lambda)u\|^2,
\end{split}\end{equation*}
and the assertion follows since $\phi'$ is bounded.

This achieves the proof of the theorem.
\end{proof}

\begin{remark}\label{trace rem}
By Theorem \ref{FKS}, $\hat v$ is locally H\"older continuous, and hence its trace$\hat v(\cdot, 0)$ on $\{y=0\}$
is well defined and it is straightforward to see that $v(\cdot,0)=u$. 

On the other hand, by a classical
interpolation theorem (\cite{LM}, Theorem 10.1),
if $\hat v$ belongs to $ W^{1,2}_{\G,  \mathrm{loc}}(\hat\Omega; \, y^adx\,dy)$, then its trace
$u$ belongs to $W^{1-a,2}_{\G, \mathrm{loc}}(\Omega)$. This shows that our assumption
$u\in W^{1-a,2}_{\G}(\G)$ is optimal as long as we are concerned with local regularity.
\end{remark}

\begin{theorem}\label{th main}
Let $-1<a<1$ and let $u\in W^{1-a,2}_\G(\G)$ be given, $u\ge 0$ on all of $\G.$ Assume $\mc L^{(1-a)/2}u = 0$
in an open set $\Omega\subset\G$. 
\\
\hphantom{xxx}Then  there exist $C,b>0$ (independent of $u$) such that
the following invariant Harnack inequality holds:
$$
\sup_{B_c(x,r)} u \le C \inf_{B_c(x,r)} u 
$$
for any metric ball $B_c(x,r)$ such that $B_c(x,b r)\subset \Omega$. 

\end{theorem}

\begin{proof}
Let $C,b$ be as in Theorem \cite{FKS}. By Theorems \ref{premain},
\cite{FKS} and by Lemma \ref{balls}, we have:
\begin{equation*}\begin{split}
\sup_{B_c(x,r)}  & u = \sup_{ \hat B_c((x,0),r)\cap \{y=0\}} \hat v
\le \sup_{ \hat B_c((x,0),r)} \hat v \le C \inf_{ \hat B_c((x,0),r)} \hat v
\\&
\le \inf_{ \hat B_c((x,0),r)\cap \{y=0\}} \hat v = \inf_{B_c(x,r)} u.
\end{split}\end{equation*}

\end{proof}

\bibliography{ferrari_franchi_Harnack_Carnot}

\begin{thebibliography}{10}

\bibitem{bogdan}
Krzysztof Bogdan.
\newblock The boundary {H}arnack principle for the fractional {L}aplacian.
\newblock {\em Studia Math.}, 123(1):43--80, 1997.

\bibitem{BLU}
Andrea Bonfiglioli, Ermanno Lanconelli, and Francesco Uguzzoni.
\newblock {\em Stratified {L}ie groups and potential theory for their
  sub-{L}aplacians}.
\newblock Springer Monographs in Mathematics. Springer, Berlin, 2007.

\bibitem{bony_courrege_priouret}
Jean-Michel Bony, Philippe Courr{\`e}ge, and Pierre Priouret.
\newblock Semi-groupes de {F}eller sur une vari\'et\'e \`a bord compacte et
  probl\`emes aux limites int\'egro-diff\'erentiels du second ordre donnant
  lieu au principe du maximum.
\newblock {\em Ann. Inst. Fourier (Grenoble)}, 18(fasc. 2):369--521 (1969),
  1968.

\bibitem{bourbaki}
Nicolas Bourbaki.
\newblock {\em \'{E}l\'ements de math\'ematique. {XXVI}. {G}roupes et
  alg\`ebres de {L}ie. {C}hapitre 1: {A}lg\`ebres de {L}ie}.
\newblock Actualit\'es Sci. Ind. No. 1285. Hermann, Paris, 1960.

\bibitem{caffarelli_silvestre}
Luis Caffarelli and Luis Silvestre.
\newblock An extension problem related to the fractional {L}aplacian.
\newblock {\em Comm. Partial Differential Equations}, 32(7-9):1245--1260, 2007.

\bibitem{caffarelli_salsa_silvestre}
Luis~A. Caffarelli, Sandro Salsa, and Luis Silvestre.
\newblock Regularity estimates for the solution and the free boundary of the
  obstacle problem for the fractional {L}aplacian.
\newblock {\em Invent. Math.}, 171(2):425--461, 2008.

\bibitem{calderon}
A.-P. Calder{\'o}n.
\newblock Inequalities for the maximal function relative to a metric.
\newblock {\em Studia Math.}, 57(3):297--306, 1976.

\bibitem{capdangar}
Luca Capogna, Donatella Danielli, and Nicola Garofalo.
\newblock The geometric {S}obolev embedding for vector fields and the
  isoperimetric inequality.
\newblock {\em Comm. Anal. Geom.}, 2(2):203--215, 1994.

\bibitem{chang_gonzalez}
Sun-Yung~Alice Chang and Mar{\'{\i}}a del~Mar Gonz{\'a}lez.
\newblock Fractional {L}aplacian in conformal geometry.
\newblock {\em Adv. Math.}, 226(2):1410--1432, 2011.

\bibitem{courrege}
Philippe Courr{\`e}ge.
\newblock Sur la forme int\'egro-diff\'erentielle des op\'erateurs de $\mathbf
  {C}^\infty_k$ dans $\mathbf {C}$ satisfaisant au principe du maximum.
\newblock In {\em S\'eminaire de {T}h\'eorie du {P}otentiel, dirig\'e par {M}.
  {B}relot, {G}. {C}hoquet et {J}. {D}eny, 1965/66, {T}ome 10, {N}o. 1}, pages
  1--38. Secr\'etariat math\'ematique, Paris, 1967.

\bibitem{FKS}
Eugene~B. Fabes, Carlos~E. Kenig, and Raul~P. Serapioni.
\newblock The local regularity of solutions of degenerate elliptic equations.
\newblock {\em Comm. Partial Differential Equations}, 7(1):77--116, 1982.

\bibitem{FF_CPDE}
Fausto Ferrari and Bruno Franchi.
\newblock A local doubling formula for the harmonic measure associated with
  subelliptic operators and applications.
\newblock {\em Comm. Partial Differential Equations}, 28(1-2):1--60, 2003.

\bibitem{folland}
Gerald~B. Folland.
\newblock Subelliptic estimates and function spaces on nilpotent {L}ie groups.
\newblock {\em Ark. Mat.}, 13(2):161--207, 1975.

\bibitem{folland_stein}
Gerald~B. Folland and Elias~M. Stein.
\newblock {\em Hardy spaces on homogeneous groups}, volume~28 of {\em
  Mathematical Notes}.
\newblock Princeton University Press, Princeton, N.J., 1982.

\bibitem{FL}
Bruno Franchi and Ermanno Lanconelli.
\newblock H\"older regularity theorem for a class of linear nonuniformly
  elliptic operators with measurable coefficients.
\newblock {\em Ann. Scuola Norm. Sup. Pisa Cl. Sci. (4)}, 10(4):523--541, 1983.

\bibitem{FS_ASNP}
Bruno Franchi and Raul Serapioni.
\newblock Pointwise estimates for a class of strongly degenerate elliptic
  operators: a geometrical approach.
\newblock {\em Ann. Scuola Norm. Sup. Pisa Cl. Sci. (4)}, 14(4):527--568, 1987.

\bibitem{FSSC_houston}
Bruno Franchi, Raul Serapioni, and Francesco Serra~Cassano.
\newblock Meyers-{S}errin type theorems and relaxation of variational integrals
  depending on vector fields.
\newblock {\em Houston J. Math.}, 22(4):859--890, 1996.

\bibitem{FSSC_step2}
Bruno Franchi, Raul Serapioni, and Francesco Serra~Cassano.
\newblock On the structure of finite perimeter sets in step 2 {C}arnot groups.
\newblock {\em J. Geom. Anal.}, 13(3):421--466, 2003.

\bibitem{FSSC_CAG}
Bruno Franchi, Raul Serapioni, and Francesco Serra~Cassano.
\newblock Regular hypersurfaces, intrinsic perimeter and implicit function
  theorem in {C}arnot groups.
\newblock {\em Comm. Anal. Geom.}, 11(5):909--944, 2003.

\bibitem{GN}
Nicola Garofalo and Duy-Minh Nhieu.
\newblock Isoperimetric and {S}obolev inequalities for
  {C}arnot-{C}arath\'eodory spaces and the existence of minimal surfaces.
\newblock {\em Comm. Pure Appl. Math.}, 49(10):1081--1144, 1996.

\bibitem{gonzalez_tan}
Maria del~Mar Gonz\'alez and Jinggang Tan.
\newblock In progress, 2012.

\bibitem{GR}
Israel~S. Gradshteyn and Iosif~M. Ryzhik.
\newblock {\em Table of integrals, series, and products}.
\newblock Academic Press [Harcourt Brace Jovanovich Publishers], New York,
  1980.
\newblock Corrected and enlarged edition edited by Alan Jeffrey, Incorporating
  the fourth edition edited by Yu. V. Geronimus [Yu. V. Geronimus] and M. Yu.
  Tseytlin [M. Yu. Tse{\u\i}tlin], Translated from the Russian.

\bibitem{Ho}
Lars H{\"o}rmander.
\newblock Hypoelliptic second order differential equations.
\newblock {\em Acta Math.}, 119:147--171, 1967.

\bibitem{hunt}
Gilbert~A. Hunt.
\newblock Semi-groups of measures on {L}ie groups.
\newblock {\em Trans. Amer. Math. Soc.}, 81:264--293, 1956.

\bibitem{landkof}
Naum~S. Landkof.
\newblock {\em Foundations of modern potential theory}.
\newblock Springer-Verlag, New York, 1972.
\newblock Translated from the Russian by A. P. Doohovskoy, Die Grundlehren der
  mathematischen Wissenschaften, Band 180.

\bibitem{levy}
Paul L{\'e}vy.
\newblock Sur les int\'egrales dont les \'el\'ements sont des variables
  al\'eatoires ind\'ependantes.
\newblock {\em Ann. Scuola Norm. Sup. Pisa Cl. Sci. (2)}, 3(3-4):337--366,
  1934.

\bibitem{LM}
J.-L. Lions and E.~Magenes.
\newblock {\em Non-homogeneous boundary value problems and applications. {V}ol.
  {I}}.
\newblock Springer-Verlag, New York, 1972.
\newblock Translated from the French by P. Kenneth, Die Grundlehren der
  mathematischen Wissenschaften, Band 181.

\bibitem{lu}
Guozhen Lu.
\newblock Weighted {P}oincar\'e and {S}obolev inequalities for vector fields
  satisfying {H}\"ormander's condition and applications.
\newblock {\em Rev. Mat. Iberoamericana}, 8(3):367--439, 1992.

\bibitem{magnani}
Valentino Magnani.
\newblock Differentiability and area formula on stratified {L}ie groups.
\newblock {\em Houston J. Math.}, 27(2):297--323, 2001.

\bibitem{masuda}
Ky{\^u}ya Masuda.
\newblock Anti-locality of the one-half power of elliptic differential
  operators.
\newblock {\em Publ. Res. Inst. Math. Sci.}, 8:207--210, 1972.

\bibitem{muckenhoupt}
Benjamin Muckenhoupt.
\newblock Weighted norm inequalities for the {H}ardy maximal function.
\newblock {\em Trans. Amer. Math. Soc.}, 165:207--226, 1972.

\bibitem{sire_valdinoci}
Yannick Sire and Enrico Valdinoci.
\newblock Fractional {L}aplacian phase transitions and boundary reactions: a
  geometric inequality and a symmetry result.
\newblock {\em J. Funct. Anal.}, 256(6):1842--1864, 2009.

\bibitem{sta}
Guido Stampacchia.
\newblock Le probl\`eme de {D}irichlet pour les \'equations elliptiques du
  second ordre \`a coefficients discontinus.
\newblock {\em Ann. Inst. Fourier (Grenoble)}, 15(fasc. 1):189--258, 1965.

\bibitem{stein}
Elias~M. Stein.
\newblock {\em Harmonic analysis: real-variable methods, orthogonality, and
  oscillatory integrals}, volume~43 of {\em Princeton Mathematical Series}.
\newblock Princeton University Press, Princeton, NJ, 1993.
\newblock With the assistance of Timothy S. Murphy, Monographs in Harmonic
  Analysis, III.

\bibitem{Stinga_Torrea}
Pablo~Ra{\'u}l Stinga and Jos{\'e}~Luis Torrea.
\newblock Extension problem and {H}arnack's inequality for some fractional
  operators.
\newblock {\em Comm. Partial Differential Equations}, 35(11):2092--2122, 2010.

\bibitem{Stinga_Zhang}
Pablo~Ra{\'u}l Stinga and Chao Zhang.
\newblock {H}arnack's inequality for fractional nonlocal equations.
\newblock {\em arXiv:1203.1518}, 2012.

\bibitem{watson}
George~N. Watson.
\newblock {\em A {T}reatise on the {T}heory of {B}essel {F}unctions}.
\newblock Cambridge University Press, Cambridge, England, 1944.

\end{thebibliography}

\bigskip

\noindent{\it
Fausto Ferrari and Bruno Franchi
\par\noindent
Dipartimento
di Matematica\par\noindent Piazza di
Porta S.~Donato 5
\par\noindent
40126 Bologna, Italy;}
\par\noindent
e-mail:
\par\noindent
fausto.ferrari@unibo.it
\par\noindent
bruno.franchi@unibo.it

\end{document}